\documentclass{jgcc} 
\pdfoutput=1
\usepackage{lastpage}
\jgccdoi{17}{2}{2}{16875}  
\jgccheading{}{\pageref{LastPage}}{}{}{Nov.~1,~2025}{Dec.~3,~2025}{}  

\keywords{test element, Turner group, co-Hopfian}

\usepackage{hyperref}
\theoremstyle{plain} 

\begin{document}

\title[A note on co-Hopfian groups and rings]{A note on co-Hopfian groups and rings}

\author[A.~M.~Gaglione]{Anthony M. Gaglione}	
\address{03 Elmwood Court, Arnold, MD 21012 U.S.A.}	
\email{agaglione@aol.com}  
\thanks{\textit{2020 Mathematics Subject Classification.} Primary 20E26,03C07; Secondary 20F19,20F05.}

\author[D.~Spellman]{Dennis Spellman}	
\address{1134 Haverford Road, Apt. A,
Crum Lynne, PA19022 U.S. A.}	
\email{dennisspellman1@aol.com}  

\begin{abstract}
  \noindent Let $p$ and $n$ be positive integers. Assume additionally that $p\neq 3$ is a prime and that $n>2$. Let $R$ be a field of characteristic $p$. A very special consequence of a result of Bunina and Kunyavskii (2023, arXiv:2308.10076) is that $SL_{n}(R)$ is co-Hopfian as a group if and only if $R$ is co-Hopfian as a ring. In this paper, we prove that if $k$ is the algebraic closure of the $2$ element field, then $SL_{2}(k)$ is a co-Hopfian group. Since this $k$ is trivially seen to be co-Hopfian as a ring our result somewhat extends that of Bunina and Kunyavskii. We apply our result to prove that the class of groups satisfying Turner's Retract Theorem (called Turner groups here) is not closed under elementary equivalence thereby answering a question posed by the authors in (2017, Comm. Algebra).
\end{abstract}

\maketitle

\hfill{\small \it In memory of Ben Fine}

\section*{Introduction}\label{S:one}

According to the account in \cite{MKS}, in 1932 Heinz Hopf
posed the question of whether or not a finitely generated group could be
isomorphic to a proper homomorphic image. Hence a group is \emph{Hopfian}
if it is not isomorphic to a proper homomorphic image. Of course every
finite group is Hopfian; so this finiteness property is of interest for
infinite groups. As the reader is no doubt well aware efforts to answer this
question were successful and there is a vast literature treating both
Hopfian and non-Hopfian groups. Clearly the Hopf property can be stated in a
universal algebraic context and it makes sense to speak, for example, of
Hopfian and non-Hopfian rings. For us a ring is an associative $\mathbb{Z}$%
-algebra with multiplicative identity $1\neq 0$. Subrings are required to
contain $1$ and homomorphisms are required to preserve $1$. More recently
the dual property has generated some interest. An algebra of some fixed type
(for us group or ring) is \emph{co-Hopfian} provided it is not isomorphic
to a proper subalgebra. Of course every finite algebra is co-Hopfian; so,
this finiteness property \ is of interest for infinite algebras. We note
that equivalent formulations of both Hopficity and co-Hopficity can be
stated interns of endomorphisms. Namely, an algebra is Hopfian if and only
if every surjective endomorphism is an automorphism and an algebra is
co-Hopfian if and only if every injective endomorphism is an automorphism.

Enough pandering! We fix notation. If $p$ is a prime $\overline{\mathbb{F}%
_{p}}$ shall be the algebraic closure of the $p$ element field. For the
remainder of this paper $k$ shall be $\overline{\mathbb{F}_{2}}$ and $G$,
standing alone, shall always be the group $SL_{2}(k)$.

\section{Examples}

We first note that if $p$ is a prime then $\overline{\mathbb{F}_{p}}$ is
co-Hopfian as a ring. Of course, since fields are simple, every endomorphism
is injective. Assume to deduce a contradiction that the endomorphism $%
\varphi :~\overline{\mathbb{F}_{p}}\rightarrow \overline{\mathbb{F}_{p}}$ is
not surjective. Suppose the image of $\varphi $ omits the element $\theta $
of $\overline{\mathbb{F}_{p}}$ . Let $f=Irr(\mathbb{F}_{p},\theta )$ be the
minimum polynomial of $\theta $ over the $p$ element field $\mathbb{F}_{p}.$
Suppose $f$ has degree $n$ and that the $n$ roots of $f$ in $\overline{%
\mathbb{F}_{p}}$ are $\theta =\theta _{1},\theta _{2},...,\theta _{n}.$
Since $\varphi $ is injective it must permute the $\theta _{i}$, so there
must be an $i,~1\leq i\leq n$, such that $\varphi (\theta _{i})=\theta $ -
contrary to hypothesis. The contradiction shows every endomorphism is
surjective; hence, as claimed, $\overline{\mathbb{F}_{p}}$ is co-Hopfian.

A similar argument shows that the multiplicative group $\overrightarrow{%
\mathbb{F}_{p}}=\overline{\mathbb{F}_{p}}\backslash \{0\}$ of $\overline{%
\mathbb{F}_{p}}$ is co-Hopfian. To see that observe that $\overrightarrow{%
\mathbb{F}_{p}}$ is locally cyclic being the direct union of the family $%
\mathbb{F}_{p^{n}}^{\ast }$ of subgroups cyclic of order $p^{n}-1$ as $n$
varies over the positive integers. Let $\phi $ be the Euler totient and,
with $N=\phi (p^{n}-1)$, let $\{\theta _{1},...,\theta _{N}\}$ be the $N$
elements of order $p^{n}-1$ in $\mathbb{F}_{p^{n}}^{\ast }$. These are
precisely the $N$ roots of the cyclotomic polynomial of degree $N$ over $%
\mathbb{F}_{P}$. Every injective homomorphism of $\overrightarrow{\mathbb{F}%
_{p}}$ must permute these. Since (taken over all $n$) these generate $%
\overrightarrow{\mathbb{F}_{p}}$ , every injective homomorphism of $%
\overrightarrow{\mathbb{F}_{p}}$ is surjective and as claimed $%
\overrightarrow{\mathbb{F}_{p}}$ is co-Hopfian as a group.

Ol'shanskii has constructed examples of groups $\Gamma $ with the following
properties.

\begin{enumerate}\item 
 $\Gamma $ is infinite
\item 
 $\Gamma $ is nonabelian
\item 
 $\Gamma \,\ $is $2$-generator
\item 
 Every proper subgroup of $\Gamma $ is cyclic.
\end{enumerate}

Let us call a group satisfying 1,2,3,4 above an \emph{Ol'shanskii group}.
Clearly no Ol'shanskii group can be isomorphic to a proper subgroup; so
evert Ol'shanskii group is co-Hopfian.

In the course of proving a rigidity result Bunina and Kunyavskii established
in \cite{BK} a preliminary proposition on Chevalley group $G(\Phi ,R)$ over local
rings $R$. Here $\Phi $ is a reduced irreducible root system of rank at
least $2$.

\begin{prop}[\cite{BK}] With the conventions and notation above and with the hypotheses
that for $\Phi =A_{2},B_{l},C_{l}.\mathbb{F}_{4}$, $2$ is a unit in $R$ and
for $\Phi =G_{2}$, $3$ is a unit in $R$, it is the case that $G(\Phi ,R)$ is
co-Hopfian as a group if and only if $R$ is co-Hopfian as a ring.
\end{prop}

Recalling that $k=\overline{\mathbb{F}_{2}}$, we thank the referee for
pointing out to us that, for $n>2$, it follows from Proposition 1 that $%
SL_{n}(k)$ is co-Hopfian.

In the next section, we shall prove that $G=SL_{2}(k)$ is co-Hopfian.

\section{The Co-Hopficity of G}

We first observe that, since $k$ is algebraically closed, the Frobenious map 
$\sigma \colon k\rightarrow k,~x\mapsto x^{2}$ is surjective$.$ Put another way,
every element of $k$ has a unique square root.

\begin{nota}
If $%
\begin{bmatrix}
a & b \\ 
c & d%
\end{bmatrix}%
\in GL_{2}(k),then\det 
\begin{bmatrix}
a & b \\ 
c & d%
\end{bmatrix}%
$ shall be its determinant $ad+bc$.
\end{nota}

We next observe that if $X\in GL_{2}(k),$then conjugation by $Y=\frac{1}{%
\sqrt{\det (X)}}\cdot X$ has the same effect as conjugation by $X;$
moreover, $Y$ lies in $G$.

It follows from this that two elements of $G$ conjugate in $GL_{2}(k)$ must
already be conjugate in $G$.

Henceforth we denote the multiplicative and additive groups of $k$ by $%
k^{\ast }$ and $k^{+}$respectively.

Every element of $GL_{2}(k)$ is conjugate to one of the following Jordan
canonical forms, unique up to the order of the blocks:%
\begin{equation*}
\begin{bmatrix}
1 & 0 \\ 
0 & 1%
\end{bmatrix}%
,~%
\begin{bmatrix}
\lambda & 1 \\ 
0 & \lambda%
\end{bmatrix}%
,~\lambda \in k^{\ast },~%
\begin{bmatrix}
\lambda & 0 \\ 
0 & \mu%
\end{bmatrix}%
,~\lambda ,\mu \in (k^{\ast })^{2}\backslash \{(1,1)\}.
\end{equation*}%
It follows that every element of $G$ is conjugate in $G$ to one of the
following Jordan canonical forms, unique up to the order of the blocks:%
\begin{equation*}
\begin{bmatrix}
1 & 0 \\ 
0 & 1%
\end{bmatrix}%
,~%
\begin{bmatrix}
1 & 1 \\ 
0 & 1%
\end{bmatrix}%
,~\lambda \in k^{\ast },~%
\begin{bmatrix}
\lambda & 0 \\ 
0 & \lambda ^{-1}%
\end{bmatrix}%
,~\lambda \in k^{\ast }\backslash \{1\}.
\end{equation*}

\begin{rem}
\begin{enumerate}\item  $\lambda =\lambda ^{-1}$ if and only if $\lambda ^{2}=1$ if and only
if $\lambda =1$.

\item Since each of taking transposes and taking inverses is an
automorphism of $G$, their composition, in either order %
\begin{equation*}
\begin{bmatrix}
a & b \\ 
c & d%
\end{bmatrix}%
\mapsto 
\begin{bmatrix}
d & c \\ 
b & a%
\end{bmatrix}%
\end{equation*}%
is an automorphism of $G$.

\item  The element $%
\begin{bmatrix}
0 & 1 \\ 
1 & 0%
\end{bmatrix}%
$ of $G$ has order $2$ and the inner automorphism
determined by this element is the inverse transpose.

\item Thus, if $\lambda \in $ $k^{\ast }\backslash \{1\}$, 
then $\lambda \neq \lambda ^{-1}$ and the matrices $%
\begin{bmatrix}
\lambda & 0 \\ 
0 & \lambda ^{-1}%
\end{bmatrix}%
$ and $%
\begin{bmatrix}
\lambda ^{-1} & 0 \\ 
0 & \lambda%
\end{bmatrix}%
$ are conjugate in $G$.
\end{enumerate}

Note that the element $%
\begin{bmatrix}
1 & 1 \\ 
0 & 1%
\end{bmatrix}%
$ of $G$ has order $2$ while, if $\lambda \in k^{\ast }\backslash \{1\}$,
then the order of $%
\begin{bmatrix}
\lambda & 0 \\ 
0 & \lambda ^{-1}%
\end{bmatrix}%
\in G$ is the order of $\lambda $ in $k^{\ast }$ and that, being a divisor
of $2^{n}-1$ for some $n$, is odd.

Note also that, for all $\lambda \in k^{\ast }\backslash \{1\},~\lambda
+\lambda ^{-1}=0$ if and only if $\lambda =\lambda ^{-1}$ if and only if $%
\lambda =1$. It follows that every element of $G$ either has odd order or
has order $2$ and, moreover, the elements of order $2$ are precisely the
nontrivial elements of trace $0$.

\end{rem}

\begin{defi}
A group $\Gamma $ which satisfies the universal sentence%
\begin{equation*}
\forall x,y,z~(((y\neq 1)\wedge (xy=yx)\wedge (yz=zy))\rightarrow (xz=zx))
\end{equation*}%
is \emph{commutative transitive}, briefly  \emph{CT}.
\end{defi}

\begin{prop}[\cite{H}] Let $\Gamma $ be a group. The following three statements are pairwise 
equivalent.
\begin{enumerate}\item 
$\Gamma$ is CT

\item  For each $g\in \Gamma \backslash \{1\}$, its
centralizer is abelian.

\item  If $M_{1}$ and $M_{2}$ are maximal abelian
subgroups in $\Gamma$, then $M_{1}\cap M_{2}=\{1\}$ 
unless $M_{1}=M_{2}$.
\end{enumerate}

\end{prop}

\begin{rem}
In any CT group,  the maximal abelian subgroups are the centralizers of
nontrivial elements.
\end{rem}

\begin{thm}[\cite{W}] Let $\Gamma $ be a finite nonsolvable group. Then $\Gamma $ is CT
if and only if it is isomorphic to $SL_{2}(\mathbb{F}_{2^{n}})$ for some
integer $n\geq 2.$
\end{thm}

Since $G$ is the direct union $\underset{\longrightarrow }{\lim }~(SL_{2}(%
\mathbb{F}_{2^{n}})$ and universal sentences are preserved in direct unions,
we have the following immediate corollary -

\begin{cor}
$G$ is CT.
\end{cor}

We make explicit some terminology we shall use going forward.

Let $X=~%
\begin{bmatrix}
x & y \\ 
z & w%
\end{bmatrix}%
\in G$. We call $X$ \emph{diagonal} if it has the form $%
\begin{bmatrix}
x & 0 \\ 
0 & x^{-1}%
\end{bmatrix}%
(x\in k^{\ast })$; \emph{off diagonal} if it has the form $%
\begin{bmatrix}
0 & y \\ 
y^{-1} & 0%
\end{bmatrix}%
(y\in k^{\ast });$ \emph{upper triangular} if it has the form $%
\begin{bmatrix}
x & y \\ 
0 & x^{-1}%
\end{bmatrix}%
(x\in k^{\ast });$ \emph{upper unitriangular} if it has the form $%
\begin{bmatrix}
1 & y \\ 
0 & 1%
\end{bmatrix}%
$; \emph{lower triangular} if it has the form $%
\begin{bmatrix}
x & 0 \\ 
z & x^{-1}%
\end{bmatrix}%
(x\in k^{\ast })$; \emph{lower unitriangular} if it has the form $%
\begin{bmatrix}
1 & 0 \\ 
z & 1%
\end{bmatrix}%
$.

We write $\Delta ,\Delta ^{\prime },U,UT,L$ and $LT$ for the sets of
diagonal, off diagonal, upper triangular, upper unitriangular, lower
triangular and lower unitriangular matrices respectively.

Now let $%
\begin{bmatrix}
s & t \\ 
u & v%
\end{bmatrix}%
\in G$ and $\lambda \in k^{\ast }$. We explicitly compute the following two
conjugations.


\begin{equation}\label{eq:conjugation1}
\begin{bmatrix}
s & t \\ 
u & v%
\end{bmatrix}%
\begin{bmatrix}
\lambda & 0 \\ 
0 & \lambda ^{-1}%
\end{bmatrix}%
\begin{bmatrix}
v & t \\ 
u & s%
\end{bmatrix}%
=%
\begin{bmatrix}
\lambda sv+\lambda ^{-1}tu & (\lambda +\lambda ^{-1})st \\ 
(\lambda +\lambda ^{-1})uv & \lambda ^{-1}sv+\lambda tu%
\end{bmatrix}%
\end{equation}

\begin{equation}
\label{eq:conjugation2}
\begin{bmatrix}
s & t \\ 
u & v%
\end{bmatrix}%
\begin{bmatrix}
1 & \lambda \\ 
0 & 1%
\end{bmatrix}%
\begin{bmatrix}
v & t \\ 
u & s%
\end{bmatrix}%
=%
\begin{bmatrix}
1+\lambda su & \lambda s^{2} \\ 
\lambda u^{2} & 1+\lambda su%
\end{bmatrix}%
\end{equation}

%
%
%

\begin{prop}
Let $\lambda \in k^{\ast }\backslash \{1\}$ and $g=%
\begin{bmatrix}
\lambda & 0 \\ 
0 & \lambda ^{-1}%
\end{bmatrix}%
$. Then $\Delta =C_{G}(g)$.
\end{prop}

\begin{proof}
We must find $%
\begin{bmatrix}
s & t \\ 
u & v%
\end{bmatrix}%
$ which conjugates $g$ to $g$. With an eye towards also determining $%
N_{G}(\Delta )$, we use the conjugation computation in Eq.~\ref{eq:conjugation1} to find more generally,
the $%
\begin{bmatrix}
s & t \\ 
u & v%
\end{bmatrix}%
$ which conjugate $\ g$ into $\Delta $. Since $\lambda \neq 1$, we have $%
\lambda +\lambda ^{-1}\neq 0$. Then $st=uv=0.$ Since $sv+tu=1$ there are two
possibilities, namely:%
\begin{equation*}
s=v=0~\text{or }t=u=0.
\end{equation*}%
If $s=v=0$, then $%
\begin{bmatrix}
s & t \\ 
u & v%
\end{bmatrix}%
\begin{bmatrix}
\lambda & 0 \\ 
0 & \lambda ^{-1}%
\end{bmatrix}%
\begin{bmatrix}
v & t \\ 
u & s%
\end{bmatrix}%
=%
\begin{bmatrix}
0 & t \\ 
t^{-1} & 0%
\end{bmatrix}%
\begin{bmatrix}
\lambda & 0 \\ 
0 & \lambda ^{-1}%
\end{bmatrix}%
\begin{bmatrix}
0 & t \\ 
t^{-1} & 0%
\end{bmatrix}%
=$ $\ 
\begin{bmatrix}
\lambda ^{-1} & 0 \\ 
0 & \lambda%
\end{bmatrix}%
\neq 
\begin{bmatrix}
\lambda & 0 \\ 
0 & \lambda ^{-1}%
\end{bmatrix}%
$ since $\lambda \neq 1$.

So $t=u=0$ and $%
\begin{bmatrix}
s & t \\ 
u & v%
\end{bmatrix}%
\begin{bmatrix}
\lambda & 0 \\ 
0 & \lambda ^{-1}%
\end{bmatrix}%
\begin{bmatrix}
v & t \\ 
u & s%
\end{bmatrix}%
=%
\begin{bmatrix}
s & 0 \\ 
0 & s^{-1}%
\end{bmatrix}%
\begin{bmatrix}
\lambda & 0 \\ 
0 & \lambda ^{-1}%
\end{bmatrix}%
\begin{bmatrix}
s^{-1} & 0 \\ 
0 & s%
\end{bmatrix}%
=%
\begin{bmatrix}
\lambda & 0 \\ 
0 & \lambda ^{-1}%
\end{bmatrix}%
.$ 
\end{proof}

\begin{cor}
If $g\in G\backslash \{1\}$ has odd order, then $C_{G}(g)$ is isomorphic to $%
k^{\ast }$.
\end{cor}

\begin{prop}
$N_{G}(\Delta )$ is generated by $\Delta ^{\prime }$. It is metabelian and
is a semidirect product of $\Delta $ by the cycle of order $2$ generated by $%
\begin{bmatrix}
0 & 1 \\ 
1 & 0%
\end{bmatrix}%
$ .
\end{prop}

\begin{proof}
By the proof of Proposition 3, $%
\begin{bmatrix}
s & t \\ 
u & v%
\end{bmatrix}%
$ conjugates $\Delta $ into itself if and only if it lies in either $\Delta $
or $\Delta ^{\prime }$. From the equivalent equations%
\begin{equation*}
\begin{bmatrix}
\lambda & 0 \\ 
0 & \lambda ^{-1}%
\end{bmatrix}%
=%
\begin{bmatrix}
0 & \lambda \\ 
\lambda ^{-1} & 0%
\end{bmatrix}%
\begin{bmatrix}
0 & 1 \\ 
1 & 0%
\end{bmatrix}%
\end{equation*}%
and%
\begin{equation*}
\begin{bmatrix}
0 & \lambda \\ 
\lambda ^{-1} & 0%
\end{bmatrix}%
=%
\begin{bmatrix}
\lambda & 0 \\ 
0 & \lambda ^{-1}%
\end{bmatrix}%
\begin{bmatrix}
0 & 1 \\ 
1 & 0%
\end{bmatrix}%
~(\lambda \in k^{\ast })
\end{equation*}%
we see that $N_{G}(\Delta )$ is generated by $\Delta ^{\prime }$ and that it
is the product of $\Delta \trianglelefteq N_{G}(\Delta )$ and $\left\langle 
\begin{bmatrix}
0 & 1 \\ 
1 & 0%
\end{bmatrix}%
\right\rangle $. Since these intersect in the identity, $N_{G}(\Delta )$ is
their semidirect product. 
\end{proof}

\begin{cor}
If $g\neq 1$ has odd order and $M=C_{G}(g)$, then $N_{G}(M)$ is metabelian
and is a semidirect product of $k^{\ast }$ by a cycle of order $2$.
\end{cor}

\begin{prop}
Let $\lambda \in k^{\ast }$ and $g=%
\begin{bmatrix}
1 & \lambda \\ 
0 & 1%
\end{bmatrix}%
$. Then $UT=C_{G}(g)$.
\end{prop}

\begin{proof}
We must find $%
\begin{bmatrix}
s & t \\ 
u & v%
\end{bmatrix}%
$ which conjugates$~g$ to $g$. With an eye towards also determining $%
N_{G}(UT)$, we use the conjugation computation in Eq.~\ref{eq:conjugation2}  to find, more generally, the 
$%
\begin{bmatrix}
s & t \\ 
u & v%
\end{bmatrix}%
$ which conjugate $g$ into $UT$. From that computation we see that this
entails $\lambda u^{2}=0$ and since $\lambda \neq 0$ we see that $u=0$; so%
\begin{equation*}
\begin{bmatrix}
s & t \\ 
u & v%
\end{bmatrix}%
=%
\begin{bmatrix}
s & t \\ 
0 & s^{-1}%
\end{bmatrix}%
\in U.
\end{equation*}%
Again, from 
 Eq.~\ref{eq:conjugation2}, we see that in order that $%
\begin{bmatrix}
s & t \\ 
u & v%
\end{bmatrix}%
=$ $%
\begin{bmatrix}
s & t \\ 
0 & s^{-1}%
\end{bmatrix}%
$ to conjugate $g$ into $g$ we must have $\lambda s^{2}=\lambda $ and since $%
\lambda \neq 0,~s^{2}=1$ and hence $s=1$. Thus, $C_{G}(g)$ consists of the $%
\begin{bmatrix}
1 & t \\ 
0 & 1%
\end{bmatrix}%
$ and so coincides with $UT$. 
\end{proof}

\begin{cor}
By taking the inverse transpose automorphism we see that, if $\lambda \in
k^{\ast }$, and $g=$ $%
\begin{bmatrix}
1 & 0 \\ 
\lambda & 1%
\end{bmatrix}%
$, then $C_{G}(g)=LT$.
\end{cor}

\begin{cor}
If $g\in G\backslash \{1\}$ has order $2$, then $C_{G}(g)$ is isomorphic to $%
k^{+}$.
\end{cor}

\begin{prop}
$N_{G}(UT)=U$. It is metabelian and a semidirect product of $UT$ by $\Delta $%
.
\end{prop}

\begin{proof}From Proposition 5, we get $N_{G}(UT)=U$. Given $%
\begin{bmatrix}
x & y \\ 
0 & x^{-1}%
\end{bmatrix}%
\in U~(x\in k\ast )$ we have $%
\begin{bmatrix}
x & y \\ 
0 & x^{-1}%
\end{bmatrix}%
=%
\begin{bmatrix}
1 & xy \\ 
0 & 1%
\end{bmatrix}%
\begin{bmatrix}
x & 0 \\ 
0 & x^{-1}%
\end{bmatrix}%
$; so, $U$ is the product of the subgroups $UT\trianglelefteq N_{G}(UT)=U$
and $\Delta $. Moreover, these subgroups intersect in the identity; so $U$
is their semidirect product. \end{proof}

\begin{rem}
The elements of order $2$ in $U$ are precisely the nontrivial elements of $%
UT $ since $x+x^{-1}=0$ implies $x=1$.
\end{rem}

\begin{cor}
By taking the inverse transpose automorphism we see that $N_{G}(LT)=L$. It
is metabelian and a semidirect product of $LT$ by $\Delta $.
\end{cor}

\begin{cor}
If $g$ has order $2$ and $M=C_{G}(g)$, then $N_{G}(M)$ is metabelian and is
a semidirect product of $k^{+}$ by $k^{\ast }$.
\end{cor}

\begin{prop}
Let $(\lambda _{1},\lambda _{2})\in (k^{\ast })^{2}$. Then $%
\begin{bmatrix}
1 & \lambda _{1} \\ 
0 & 1%
\end{bmatrix}%
$ and $%
\begin{bmatrix}
1 & 0 \\ 
\lambda _{2} & 1%
\end{bmatrix}%
$ do not commute.
\end{prop}

\begin{proof}
The centralizers of $%
\begin{bmatrix}
1 & \lambda _{1} \\ 
0 & 1%
\end{bmatrix}%
$ and $%
\begin{bmatrix}
1 & 0 \\ 
\lambda _{2} & 1%
\end{bmatrix}%
$ are $UT$ and $LT$ respectively and these have trivial intersection.\end{proof}

Every element of order $2$ in $G$ is conjugate to $%
\begin{bmatrix}
1 & 1 \\ 
0 & 1%
\end{bmatrix}%
$ which lies in the subgroup $SL_{2}(\mathbb{F}_{4});$ $SL_{2}(\mathbb{F}%
_{4})$ is a simple group of order $60$. Every simple group of order $60$ is
isomorphic to the alternating group $A_{5}$ (see \cite{D}). Since $A_{5}$ is
generated by $3$-cycles, $%
\begin{bmatrix}
1 & 1 \\ 
0 & 1%
\end{bmatrix}%
$ is the product of elements of order $3$. Thus $G$ is generated by its
elements of odd order. Every element of odd order is conjugate to an element
of $\Delta $. Given $\lambda \in k^{\ast }$ we have%
\begin{equation*}
\begin{bmatrix}
0 & \lambda \\ 
\lambda ^{-1} & 0%
\end{bmatrix}%
\begin{bmatrix}
0 & 1 \\ 
1 & 0%
\end{bmatrix}%
=%
\begin{bmatrix}
\lambda & 0 \\ 
0 & \lambda ^{-1}%
\end{bmatrix}%
;
\end{equation*}%
so, every element of odd order is the product of two elements of order $2$
and $G$ is generated by its elements of order $2$. These are conjugates of $%
\begin{bmatrix}
1 & 1 \\ 
0 & 1%
\end{bmatrix}%
$ and, by 
 Eq.~\ref{eq:conjugation2}, have the form $%
\begin{bmatrix}
1+su & s^{2} \\ 
u^{2} & 1+su%
\end{bmatrix}%
$.

These depend on $s$ and $u$ only; so, if $s\neq 0$, we can take the
conjugating matrix to be the lower triangular matrix $%
\begin{bmatrix}
s & 0 \\ 
u & s^{-1}%
\end{bmatrix}%
$ while if $s=0$ then $%
\begin{bmatrix}
1+su & s^{2} \\ 
u^{2} & 1+su%
\end{bmatrix}%
=%
\begin{bmatrix}
1 & 0 \\ 
u^{2} & 1%
\end{bmatrix}%
$ is itself lower triangular.

(Note that the above is an arbitrary element of $\ LT$ as the Foebenius map $%
k\rightarrow k,~x\longmapsto x^{2}$ is an automorphism.) It follows that $G$
is generated by $%
\begin{bmatrix}
1 & 1 \\ 
0 & 1%
\end{bmatrix}%
$ and $L$. Now $%
\begin{bmatrix}
1 & 0 \\ 
1 & 1%
\end{bmatrix}%
\in L$ and $%
\begin{bmatrix}
0 & 1 \\ 
1 & 0%
\end{bmatrix}%
$ conjugates $%
\begin{bmatrix}
1 & 0 \\ 
1 & 1%
\end{bmatrix}%
$ to $%
\begin{bmatrix}
1 & 1 \\ 
0 & 1%
\end{bmatrix}%
$; so $G$ is generated by $%
\begin{bmatrix}
0 & 1 \\ 
1 & 0%
\end{bmatrix}%
$ and $L$. It is therefore generated by the larger set $\Delta ^{\prime }$
and $L$ and hence generated by the subgroups $N_{G}(\Delta )$ and $L$.

Now suppose \ $\varphi\colon G\rightarrow G$ is an injective endomorphism. Let $%
\theta \in k$ be a primitive cube root of unity. Let $g$ be the order $3\,\ $%
element $%
\begin{bmatrix}
\theta & 0 \\ 
0 & \theta ^{-1}%
\end{bmatrix}%
$.

Then $\varphi (g)$ has order $3$ so is a conjugate to $%
\begin{bmatrix}
\theta & 0 \\ 
0 & \theta ^{-1}%
\end{bmatrix}%
$.

Let the inner automorphism $\alpha $ conjugate $\varphi (g)$ to $g$. Then $%
\alpha \varphi $ is also an injective endomorphism. Since $\alpha \varphi $
fixes $g$ it must map $\Delta =C_{G}(g)$ into $\Delta $ and so restricts to
an injective endomorphism of $\Delta $. Now $\Delta $ is isomorphic to $%
k^{\ast }$ which, by the example $\overrightarrow{\mathbb{F}_{p}}$, of
Section 2 with $p=2$ is co-Hopfian. Hence, $\alpha \varphi $ restricts to an
automorphism of $\Delta $. Then $\alpha \varphi $ maps $N_{G}(\Delta )$ into 
$N_{G}(\Delta )$. Now $\alpha \varphi \left( 
\begin{bmatrix}
0 & 1 \\ 
1 & 0%
\end{bmatrix}%
\right) $ maps to an element of order $2$ so must lie in $\Delta ^{\prime }$%
. Say $\alpha \varphi \left( 
\begin{bmatrix}
0 & 1 \\ 
1 & 0%
\end{bmatrix}%
\right) =%
\begin{bmatrix}
0 & \lambda \\ 
\lambda ^{-1} & 0%
\end{bmatrix}%
$. Since $\Delta \leq \text{Im}(\alpha \varphi )$, $%
\begin{bmatrix}
\lambda ^{-1} & 0 \\ 
0 & \lambda%
\end{bmatrix}%
\begin{bmatrix}
0 & \lambda \\ 
\lambda ^{-1} & 0%
\end{bmatrix}%
=%
\begin{bmatrix}
0 & 1 \\ 
1 & 0%
\end{bmatrix}%
$ lies in $\text{Im}(\alpha \varphi ).$ Now $N_{G}(\Delta )$ is generated by 
$%
\begin{bmatrix}
0 & 1 \\ 
1 & 0%
\end{bmatrix}%
$ and $\Delta $ ; so $\alpha \varphi $ restricts to an automorphism of $%
N_{G}(\Delta )$.

What does $\alpha \varphi $ do on $L$? $L$ is generated by $\Delta $ and
elements of order $2$ which are conjugated by $\Delta $ into commuting
elements of order $2$. What elements of order $2$ are conjugated by $\Delta $
into elements which commute with the original? As we have seen before an
arbitrary element of order $2$ has the form%
\begin{equation*}
\begin{bmatrix}
1+su & s^{2} \\ 
u^{2} & 1+su%
\end{bmatrix}%
\end{equation*}%
with $s\neq 0$ or $u\neq 0$. Let $\lambda \in k^{\ast }\backslash \{1\}$.
Computing%
\begin{equation*}
\begin{bmatrix}
\lambda & 0 \\ 
0 & \lambda ^{-1}%
\end{bmatrix}%
\begin{bmatrix}
1+su & s^{2} \\ 
u^{2} & 1+su%
\end{bmatrix}%
\begin{bmatrix}
\lambda ^{-1} & 0 \\ 
0 & \lambda%
\end{bmatrix}%
\end{equation*}%
we get%
\begin{equation*}
\begin{bmatrix}
1+su & \lambda ^{2}s^{2} \\ 
\lambda ^{-2}u^{2} & 1+su%
\end{bmatrix}%
.
\end{equation*}%
When does this commute with $%
\begin{bmatrix}
1+su & s^{2} \\ 
u^{2} & 1+su%
\end{bmatrix}%
$?%
\begin{equation*}
\begin{bmatrix}
1+su & s^{2} \\ 
u^{2} & 1+su%
\end{bmatrix}%
\begin{bmatrix}
1+su & \lambda ^{2}s^{2} \\ 
\lambda ^{-2}u^{2} & 1+su%
\end{bmatrix}%
=%
\begin{bmatrix}
(1+su)^{2}+\lambda ^{-2}s^{2}u^{2} & \ast \\ 
\ast & \ast%
\end{bmatrix}%
\end{equation*}

\begin{equation*}
\begin{bmatrix}
1+su & \lambda ^{2}s^{2} \\ 
\lambda ^{-2}u^{2} & 1+su%
\end{bmatrix}%
\begin{bmatrix}
1+su & s^{2} \\ 
u^{2} & 1+su%
\end{bmatrix}%
=%
\begin{bmatrix}
(1+su)^{2}+\lambda ^{2}s^{2}u^{2} & \ast \\ 
\ast & \ast%
\end{bmatrix}%
\end{equation*}%
so $\lambda ^{2}s^{2}u^{2}=\lambda ^{-2}s^{2}u^{2}$ and $\lambda su=\lambda
^{-1}su$ and $(\lambda +\lambda ^{-1})su=0$. Since $\lambda \neq 1,\lambda
+\lambda ^{-1}\neq 0$. Therefore, either $s=0$ or $u=0$. It follows that%
\begin{equation*}
\begin{bmatrix}
1+su & s^{2} \\ 
u^{2} & 1+su%
\end{bmatrix}%
=%
\begin{bmatrix}
1 & 0 \\ 
u^{2} & 1%
\end{bmatrix}%
\in LT
\end{equation*}%
or%
\begin{equation*}
\begin{bmatrix}
1+su & s^{2} \\ 
u^{2} & 1+su%
\end{bmatrix}%
=%
\begin{bmatrix}
1 & s^{2} \\ 
0 & 1%
\end{bmatrix}%
\in UT.
\end{equation*}%
By Proposition 7 we cannot have nontrivial instances of both $LT$ and $UT$
in the image of $\alpha \varphi $ on $L$. Now let $\beta $ be the identity
automorphism if the image of $\alpha \varphi $ on $L$ contains nontrivial
elements of $LT$ and be the inverse transpose automorphism (conjugation by $%
\begin{bmatrix}
0 & 1 \\ 
1 & 0%
\end{bmatrix}%
$) if the image of $\alpha \varphi $ on $L$ contains nontrivial elements of $%
UT$.


The map $\varphi$ 
will be an automorphism if and only if $\beta \alpha \varphi 
$ is. The map $\beta \alpha \varphi $ leaves $N_{G}(\Delta )$ alone and maps $L$
into $L$.%
\begin{equation*}
\beta \alpha \varphi \left( 
\begin{bmatrix}
1 & 0 \\ 
1 & 1%
\end{bmatrix}%
\right) =%
\begin{bmatrix}
1 & 0 \\ 
\lambda & 1%
\end{bmatrix}%
\end{equation*}%
for some $\lambda \in k^{\ast }$. Let $z\in k^{\ast }$ be arbitrary. Then 
\begin{equation*}
\begin{bmatrix}
\sqrt{\lambda z^{-1}} & 0 \\ 
0 & \sqrt{\lambda ^{-1}z}%
\end{bmatrix}%
\begin{bmatrix}
1 & 0 \\ 
\lambda & 1%
\end{bmatrix}%
\begin{bmatrix}
\sqrt{\lambda ^{-1}z} & 0 \\ 
0 & \sqrt{\lambda z^{-1}}%
\end{bmatrix}%
=%
\begin{bmatrix}
1 & 0 \\ 
z & 1%
\end{bmatrix}%
\end{equation*}%
so an arbitrary element of $LT$ lies in the image of $\beta \alpha \varphi $
on $L$.

Since $L$ is generated by $\Delta $ and $LT$ , $\beta \alpha \varphi $
restricts to an automorphism of $L$. Since $N_{G}(\Delta )$ and $L$, $\beta
\alpha \varphi $ is an automorphism of $G$ and thus $\varphi $ is also.

We have proven:

\begin{thm}
$G$ is co-Hopfian.
\end{thm}

\section{An application to Turner groups}

\begin{defi}
An element $g$ of a group $\Gamma $ is a \emph{test element }provided
every endomorphism of $\Gamma $ which fixes $g$ is an automorphism.
\end{defi}

\begin{defi}[\cite{OT}] A hyperbolic group $\Gamma $ is \emph{stably hyperbolic \ }if
for each endomorphism $\varphi\colon \Gamma \rightarrow \Gamma $ and each
positive integer $n$, there is an integer $m\geq n$ such that $\varphi
^{m}(\Gamma )$ is hyperbolic.
\end{defi}

\begin{rem}
Finite groups and finitely generated free groups are stably hyperbolic.
\end{rem}

In \cite{OT} O'Neil and Turner proved

\begin{thm}
In any stably hyperbolic group an element is a test element if and only if
it lies in no proper retract.
\end{thm}

\begin{rem}
Clearly lying in no proper retract is a necessary condition to be a test
element. So the real theorem is that in stably hyperbolic groups this
condition is sufficient.
\end{rem}

\begin{defi}[\cite{FGLS}] A group $G$ is a \emph{Turner group \ }provided every
element which is excluded by every proper retract is a test element.
\end{defi}

It was shown in \cite{FGLS} that the class of Turner groups is not the model
class of any set of first order sentences in a language appropriate for
group theory. In that paper it was posed as an open question whether or not
the class of Turner groups satisfies the weaker condition of closure under
elementary equivalence. We shall see below that it follows from the
co-Hopficity of $G$ the class of Turner groups is not closed under
elementary equivalence.

\begin{prop}[\cite{FGLS}] Every co-Hopfian simple group is a Turner group.
\end{prop}
\begin{proof}
Let $\Gamma $ be a co-Hopfian simple group. Since $\Gamma $
is simple the only proper retract is the trivial group $\{1\}$. Let $g\in
\Gamma \backslash \{1\}$ and let the endomorphism $\varphi\colon \Gamma
\rightarrow \Gamma $ fix $g$. Then $g\notin Ker(\varphi )$ so $Ker(\varphi
)\neq \Gamma $ and hence $Ker(\varphi )=\{1\}$. Therefore, $\varphi $ is
injective. By co-Hopficity, $\varphi $ is an automorphism. 
\end{proof}

\begin{cor}
$G$ is a Turner group.
\end{cor}

\begin{proof}
$G=SL_{2}(k)=PSL_{2}(k)$ is a co-Hopfian simple group.\end{proof}

\begin{prop}[\cite{FGLS}] Let $K$ be a field of characteristic $2$ which is not
co-Hopfian as a ring. Then $SL_{2}(K)$ is not a Turner group.
\end{prop}

\begin{proof}
Let $K_{o}$ be a proper subfield of $K$ and let $\varphi
\colon K\rightarrow K_{0}$ be an isomorphism. Then $\varphi $ induces an
endomorphism $\overline{\varphi }\colon SL_{2}(K)\rightarrow SL_{2}(K)$ via%
\begin{equation*}
\overline{\varphi }~%
\begin{bmatrix}
a & b \\ 
c & d%
\end{bmatrix}%
=%
\begin{bmatrix}
\varphi (a) & \varphi (b) \\ 
\varphi (c) & \varphi (d)%
\end{bmatrix}%
\end{equation*}%
$\overline{\varphi }$ fixes every element of $SL_{2}(\mathbb{F}_{2})$ but $%
\overline{\varphi }$ is not an automorphism since, if $\theta \in
K\backslash K_{0}$ then e.g. $%
\begin{bmatrix}
1 & \theta \\ 
0 & 1%
\end{bmatrix}%
$ does not lie in the image of $\overline{\varphi }$.\end{proof}

Now let $t$ be transcendental over $k=\overline{\mathbb{F}_{2}}$ . Let $K$
be the algebraic closure of the transcendental extension $k(t)$ of $k$.
Since the theory of algebraically closed fields of characteristic $2$ is
complete (a consequence of \cite[Chapter 9, Corollary 1.11]{BS}) the fields $%
k $ and $K$ are elementarily equivalent. By the Keisler-Shelah Ultrapower
Theorem \cite{S}; there is a nonempty index set $I$ and an ultrafilter $D$ on $I$
such that the ultrapowers $^{\ast }K=K^{I}/D$ and $^{\ast }k=k^{I}/D$ are
isomorphic. Then $^{\ast }K$ is isomorphic to a proper subfield $^{\ast }k$
and then by Proposition 9, $SL_{2}(^{\ast }K)$ is not a Turner group.

Now we have isomorphisms%
\begin{eqnarray*}
SL_{2}(^{\ast }K) &\cong &SL_{2}(K)^{I}/D \\
&\cong &SL_{2}(k)^{I}/D \\
&=&^{\ast }SL_{2}(k)
\end{eqnarray*}%
and the class of Turner groups is not closed under ultrapowers hence not
closed under elementary equivalence. \

\begin{rem}
The argument also shows that the classes of co-Hopfian rings and groups are
not closed under elementary equivalence.
\end{rem}

\bibliographystyle{plain}
\bibliography{refs}

\begin{thebibliography}{1}

\bibitem{BS}
J.~L. Bell and A.~B. Slomson.
\newblock {\em Models and ultraproducts: {A}n introduction}.
\newblock North-Holland Publishing Co., Amsterdam-London, 1969.

\bibitem{BK}
Elena Bunina and Boris Kunyavskii.
\newblock Sha-rigidity of chevalley groups over local rings, 2023.

\bibitem{D}
John~D. Dixon.
\newblock {\em Problems in group theory}.
\newblock Blaisdell Publishing Co. [Ginn and Co.], Waltham, Mass.-Toronto,
  Ont.-London, 1967.

\bibitem{FGLS}
Benjamin Fine, Anthony Gaglione, Seymour Lipschutz, and Dennis Spellman.
\newblock On {T}urner's theorem and first-order theory.
\newblock {\em Comm. Algebra}, 45(1):29--46, 2017.

\bibitem{H}
Nancy Harrison.
\newblock Real length functions in groups.
\newblock {\em Trans. Amer. Math. Soc.}, 174:77--106, 1972.

\bibitem{MKS}
Wilhelm Magnus, Abraham Karrass, and Donald Solitar.
\newblock {\em Combinatorial group theory}.
\newblock Dover Publications, Inc., Mineola, NY, second edition, 2004.
\newblock Presentations of groups in terms of generators and relations.

\bibitem{OT}
John~C. O'Neill and Edward~C. Turner.
\newblock Test elements and the retract theorem in hyperbolic groups.
\newblock {\em New York J. Math.}, 6:107--117, 2000.

\bibitem{S}
Saharon Shelah.
\newblock Every two elementarily equivalent models have isomorphic ultrapowers.
\newblock {\em Israel J. Math.}, 10:224--233, 1971.

\bibitem{W}
Yu-Fen Wu.
\newblock Groups in which commutativity is a transitive relation.
\newblock {\em J. Algebra}, 207(1):165--181, 1998.

\end{thebibliography}

\end{document}